\pdfoutput=1

\documentclass[11pt, oneside]{amsart}
\usepackage[a4paper, total={5.5in, 9in}]{geometry}  % see geometry.pdf on how to lay out the page. There's lots.
\usepackage{lipsum} 
\usepackage{sidecap}
\usepackage{float}
\usepackage{tgtermes}  %FONT STYLE Times
\usepackage{color}

\usepackage{enumitem}
\newlist{myQuoteEnumerate}{enumerate}{2}% Set max nesting depth
\setlist[myQuoteEnumerate,1]{label=(\alph*)}% Use numbers for level 1
\setlist[myQuoteEnumerate,2]{label=(\alph*)}%   Use letters for level 2

\usepackage[usenames,dvipsnames,svgnames,table]{xcolor}
\definecolor{awesome}{rgb}{1.0, 0.13, 0.32}

\definecolor{airforceblue}{rgb}{0.36, 0.54, 0.66}

\usepackage[dvipsnames]{xcolor}
\usepackage{amsmath,amssymb,latexsym}
\usepackage[utf8]{inputenc}
\usepackage[english]{babel}
\usepackage{booktabs}
\newtheorem{thm}{Theorem}[section]
\newtheorem{thmm}{Theorem}[section]
\newtheorem{cor}{Corollary}[thmm]
\newtheorem{lem}{Lemma}[section]

\newtheorem{defi}{Definition}[section]

\newtheorem{remark}{Remark}[section]

\usepackage{siunitx}
\usepackage{xcolor}
\usepackage{booktabs,colortbl, array}
\usepackage{pgfplotstable}
\definecolor{rulecolor}{RGB}{0,71,171}
\definecolor{tableheadcolor}{gray}{0.92}
\numberwithin{equation}{section}

\newcommand\quotient[2]{
        \mathchoice
            {% \displaystyle
                \text{\raise1ex\hbox{$#1$}\Big/\lower1ex\hbox{$#2$}}%
            }
            {% \textstyle
                #1\,/\,#2
            }
            {% \scriptstyle
                #1\,/\,#2
            }
            {% \scriptscriptstyle  
                #1\,/\,#2
            }
    }

\definecolor{aurometalsaurus}{rgb}{0.43, 0.5, 0.5}
\definecolor{darkjunglegreen}{rgb}{0.1, 0.14, 0.13}
\definecolor{coolblack}{rgb}{0.0, 0.18, 0.39}
\definecolor{cobalt}{rgb}{0.0, 0.28, 0.67}

\usepackage{hyperref}
\hypersetup{
     colorlinks   = true,
     citecolor    = blue
}

\title{Hausdorff dimension of Julia sets of unicritical correspondences }
\author{Carlos Siqueira \\  {  \tiny \today}}

\address{Departamento de Matem\'atica, IME-UFBA.  Salvador -- BA, Brazil. CEP 40170-115
}
\email[C.~Siqueira]{carlos.siqueira@ufba.br}
\begin{document}

\hypersetup{linkcolor=cobalt}

\begin{abstract} We show that if $\beta>1$ is a rational number and the Julia set  $J$ of the holomorphic correspondence $z^{\beta}+c$ is a locally eventually onto hyperbolic repeller, then the Hausdorff dimension  of $J$ is bounded from above by the zero of the associated  pressure function. As a consequence,  we conclude that the Julia set of the correspondence has zero Lebesgue measure for parameters close to zero, whenever $q^2<p$ and $\beta=p/q$ in lowest terms. 
\end{abstract}

\maketitle

\keywords{MSC-class 2020:  37F05 (Primary)   37D35 (Secondary). }

%\tableofcontents

\section{Introduction}

%%%

 \noindent This paper is concerned with the {oneparameter} family of  holomorphic correspondences {acting on the complex plane} \begin{equation}\label{woekg} (w-c)^q=z^p,\end{equation} where $p>q \geq 2$ {are integers.}  {A holomorphic correspondence, {as defined in this paper,} is a relation $z\mapsto w$ determined by a polynomial equation $P(z,w)=0$ in two complex variables. The family \eqref{woekg} {has bidegree $(p,q)$; this means  that \eqref{woekg} defines a multifunction {that maps} every $z\neq0$ to $q$ different values of $w$ and its inverse maps every $w\neq c$ to $p$ values of $z.$} {We shall establish an upper bound for the Hausdorff dimension of certain Julia sets in this family (Theorem A), thereby showing that the Julia set of \eqref{woekg} has zero Lebesgue measure for parameters $c$ close to zero when $q^2 <p$ (Theorem B).}

The family  \eqref{woekg} can be regarded as a generalisation of the quadratic family $z^2+c$, for if $\beta=p/q$ in lowest terms, then \eqref{woekg} is the multifunction $\mathbf{f}_c(z)=z^{\beta}+c$, in the sense that $\mathbf{f}_c(z)$ is the set of all $w$ satisfying  \eqref{woekg}. There is another way of generalising the quadratic family in the context of holomorphic correspondences using matings. Indeed, the space of holomorphic correspondences includes  all possible Kleinian groups and rational maps, and some correspondences concentrate in a single expression $P(z,w)=0$ the dynamics of a polynomial map and a Kleinian group.   This is the case of the oneparameter family of holomorphic correspondences $\mathcal{F}_a$
%\begin{equation} \label{bncs} \left ( \frac{aw-1}{w-1} \right)^2 + \left (\frac{aw-1}{w-1} \right) \left (\frac{az+1}{z+1} \right)  + \left( \frac{az+1}{z+1}\right)^2=3, 
%\end{equation}
 introduced by Bullett and Penrose nearly thirty years ago in \cite{Bullett1994}, when they discovered that   if $4\leq a \leq 7$, then the correspondence $\mathcal{F}_a$ is a mating between some quadratic map $z^2 +c$ and the modular group $\operatorname{PSL}(2, \mathbb{Z}).$ In the same paper \cite{Bullett1994}, the authors conjectured that the connectedness locus of the family $\mathcal{F}_a$ is homeomorphic to the Mandelbrot set. Recently, Bullett and Lomonaco  \cite{BL19}  have shown that  {\it $\mathcal{F}_a$   is a mating between some parabolic map and the modular group, for every parameter $a$ in the connectedness locus of the family $\mathcal{F}_a$.} Moreover, they have developed a strategy to prove that the connectedness locus of the family of matings is indeed homeomorphic to the Mandelbrot set, see \cite[page 4]{BL17} and \cite{BL20}.  Holomorphic correspondences also appear in many other contexts; indeed,  Lee, {Lyubich,} Makarov and Mukherjee \cite{Mukherjee} have  investigated the dynamics {oneparameter} families of Schwarz reflections which give rise to anti-holomorphic correspondences that are, in a suitable sense, matings of anti-rational maps with the abstract modular group $\mathbb{Z}_2 * \mathbb{Z}_3.$

\subsection{Hausdorff dimension} % Hyperbolic maps are conjecturally\footnote{See \cite{BLS} for a brief explanation of the Fatou and MLC conjectures.} open and dense in this family, and under such hypothesis, the Lebesgue measure of $J(f_c)$ is always zero when $c$ is in the interior of $\mathcal{M}_d.$
  The Julia set $J(\mathbf{f}_c)$ of $\mathbf{f}_c(z)=z^{\beta} +c$, when $\beta>1$,  is defined as the closure of all repelling cycles (see section \ref{qpdkg});  
it is always the projection of a solenoid in $\mathbb{C}^2$ when the parameter $c$ is close to zero, as described by Siqueira and Smania in \cite{SS17}.
Simon \cite{Simon97} has derived an explicit formula for the Hausdorff dimension of the Smale-Williams solenoid which relies on the zero of the pressure function, see \cite[page 1224]{Simon97}.  Before Simon,  the pioneer work of Bowen  \cite{Bowen1979} on quasi-Fuchsian groups was the first to establish the formula $P(t\phi)=0$, relating  the Hausdorff dimension to the unique zero of the pressure function.    Similarly, in the early eighties Ruelle \cite{Ruelle}  proved that  if the Julia set {$J(f)$} of a rational function $f$ is hyperbolic, then the Hausdorff dimension of {$J(f)$} depends real analytically on $f.$ The strategy used by Ruelle consists of: 

(I) showing that the Hausdorff dimension is given by the \emph{Bowen's formula:}  $P(t\phi)=0,$ where $t=\dim_H J$ and  $\phi(z)=-\log |f'(z)|$ is the geometric potential; and

 (II) proving that {$f \mapsto t(f)$} is real analytic, where {$t=t(f)$} is implicitly given by $P(t\phi)=0.$

 In the context of holomorphic correspondences, we have the following result.

 \begin{thm}[Theorem \ref{kfieg}] \label{bjdq}Suppose $J(\mathbf{f}_c)$ is a locally eventually onto hyperbolic repeller and let $t_c$ be the unique zero of the pressure function. Then $\operatorname{dim_{H}} J(\mathbf{f}_c) \leq t_c. $

 \end{thm}

 The Bowen parameter $t_c$ in Theorem \ref{bjdq} comes from an expanding and topologically mixing map $f_c: J(f_c) \to J(f_c)$ acting on a `Julia set' $J(f_c) \subset \mathbb{C}^2$ -- see Theorem \ref{igdg}.
 % -- and for this reason the upper bound $t_c$ is not exactly $\dim_H J(\mathbf{f}_c)$ in many cases. %
  The shape of $J(f_c)$ is similar to that of the Smale-Williams solenoid for parameters close to zero, and the projection of $J(f_c)$ is always the Julia set $J(\mathbf{f}_c)$  in the plane \cite{SS17}.  The sets $J(f_c)$ are related by a holomorphic motion in $\mathbb{C}^2$ (section \ref{bndhw}). 
  Since $J(f_c) \subset \mathbb{C}^2$ moves holomorphically, we believe that the Hausdorff dimension of $J(f_c)$  depends real analytically, or at least continuously on $c.$ However, this problem is still unsolved. (It should be noticed that even though  $t_c$ comes from the dynamics of $f_c$ on $J(f_c) \subset \mathbb{C}^2,$ in this paper the parameter $t_c$  is used to estimate the Hausdorff dimension of the Julia set $J(\mathbf{f}_c)$ \emph{in the plane}).

   The estimate provided by Theorem \ref{bjdq} can be used to derive the following result.

  \begin{thm}[Corollary \ref{hjeeh}] \label{qpoig}If $c$ is sufficiently close to zero and $q^2<p$, then the solenoidal Julia set of $z^{p/q} +c$ has zero Lebesgue measure.

\end{thm}

See \cite[page 3106]{SS17} for a typical figure of the solenoidal Julia set of Theorem \ref{qpoig}.
 
% \noindent Zinsmeister   \cite{Zinsmeister} presents an alternative proof of the real analyticity of the Hausdorff dimension using holomorphic motions instead of zeta functions, but the strategy is essentially the same of Ruelle's and consists of the steps (I) and (II) above.  

\subsection{Hyperbolic components} The \emph{connectedness locus} of the family $z^{\beta}+c$, denoted by $\mathcal{M}_{\beta}$ (for $\beta >1$ and $\beta \in \mathbb{Q}),$ is the set of all parameters $c$ for which the Julia set of $z^{\beta}+c$ is connected.   We define $\mathcal{M}_{\beta,0}$  as the set of $c$ for which zero has at least one bounded forward orbit under $z^{\beta}+c.$ {Siqueira has shown \cite[Theorem 2.3]{Rigidity}} that $\mathcal{M}_{\beta}$  contains $\mathcal{M}_{\beta,0}.$ Therefore,   the Julia set of $z^{\beta}+c$ is connected whenever the critical point  has at least one bounded forward orbit under $\mathbf{f}_c.$ 

\noindent For the quadratic family we have $\beta=2$ and the definitions  of $\mathcal{M}_{2}$ and $\mathcal{M}_{2,0}$ coincide with the Mandelbrot set. For some non-integer values of $\beta$, the parameter space $\mathcal{M}_{\beta} - \mathcal{M}_{\beta,0}$ {is known to be nonempty } and generates an intriguing class of Julia sets named Carpets: they are hyperbolic, connected, and {have infinitely many holes.} In spite of being hyperbolic, Carpets seem to have positive area. See \cite[section 3]{BLS} for more details and figures.

A {special} version of the Fatou conjecture for polynomial maps states that the interior of $\mathcal{M}_d$ consists of hyperbolic parameters. One implication of this conjecture is: \emph{the Julia set of every polynomial $z^d+c$ has zero Lebesgue measure if $c$ is in the interior of $\mathcal{M}_d$.} The Hausdorff dimension $d(c)$ of the Julia set of the polynomial $z^d +c$ is real analytic on every hyperbolic component; in particular, it is real analytic on $\mathbb{C} - \mathcal{M}_d.$  In the boundary of $\mathcal{M}_d,$ the Hausdorff dimension $d(c)$  is not even continuous at semihyperbolic parameters. However, Rivera-Letelier  \cite{rivera-letelier2001} has established some sort of continuity of $d(c)$ at semihyperbolic parameters $c_0$ in the boundary of the multibrot set $\mathcal{M}_d$, proving that {$d(c_n) \to d(c_0),$} whenever $c_n$ converges to $c_0$ in an appropriate way (see \cite{rivera-letelier2001} for more details). 

Little is known about the connectedness locus $\mathcal{M}_{\beta}$ of the family $z^{\beta}+c.$ One possible way to start the investigation has been presented in \cite{Rigidity}: study the dynamics of $z^{\beta}+c$ when $c$ is close to a centre. Recall from Douady and Hubbard \cite{DH84, DH85}  that the  Mandelbrot set has infinitely many hyperbolic components  $U$, each of which encoded by a centre $c\in U$; the centre is the only parameter in $U$ for which  the orbit of the critical point of $z^2+c$ is a cycle. For the correspondence $z^{\beta}+c,$ the parameter $a$ is a \emph{simple centre} if only  one forward orbit of zero $(0, z_1, z_2, \ldots)$  is periodic, and any other orbit $(0, w_1, w_2, \ldots)$ diverges to infinity ({see Definition \ref{hgdw};} it is not necessary to compute all orbits to test if $a$ is a simple centre. Indeed, the basin of attraction of $\infty$ contains a forward invariant  disk $|z|>R$, and therefore we have to check only finitely many iterates).

There exists an open set $H_{\beta}$  containing the complement of $\mathcal{M}_{\beta,0}$ and every simple centre such that, for any $c$ in $H_{\beta},$  $z^{\beta}+c$ is hyperbolic and its Julia set is stable by means of branched holomorphic motions; moreover, $J(\mathbf{f}_c)$ is a locally eventually onto (LEO) hyperbolic repeller, whenever $c\in H_{\beta}.$  { (See Definition 4.2, Corollary 5.7.1 and Theorems 5.8 and 5.10 of \cite{Rigidity}).   }

Combining these facts with Theorem \ref{bjdq} we have the following result. 

\begin{thm}[Corollaries \ref{abc} and \ref{abcxz}] If $c$ is sufficiently close to a simple centre, or if $c$ belongs to the complement of $\mathcal{M}_{\beta,0}$, then $$\dim_H J(\mathbf{f}_c) \leq t_c.$$
\end{thm}

\section{Hyperbolic sets}  \label{qpdkg}

Consider the  holomorphic correspondence \eqref{woekg}, i.e., the relation $z\mapsto w$ determined by the polynomial equation $(w-c)^{q}=z^p.$ This correspondence shall be denoted by $\mathbf{f}_c.$ The integers $p> q \geq 2$ are fixed. {Therefore, $(\mathbf{f}_c)_{c\in \mathbb{C}}$ is a oneparameter family of holomorphic correspondences.  } 

\subsection{Preliminary definitions}  {If $z$ and $w$ are related by \eqref{woekg}, then we say that $w$ is an \emph{image} of $z$ and write $z \mapsto w.$}   Every sequence $(z_i)_0^{\infty}$ of the plane where $z_{i+1}$ is an image of $z_i$ is an \emph{orbit of the correspondence.} An orbit is a \emph{cycle of period $n$} if $z_n=z_0,$  and $z_i \neq z_0$ if $0<i<n.$ 

The set of images of a point $z$ under the correspondence \eqref{woekg}  is denoted by $\mathbf{f}_{c}(z).$ Similarly, $\mathbf{f}_c(A)$ is the union of all $\mathbf{f}_c(z)$ when $z$ belongs $A.$ The inverse image set $\mathbf{f}_{c}^{-1}(A)$ is defined by the set of all $z$  which has at least one image in $A.$

 If $z_1$ is an image of $z_0$ under the correspondence \eqref{woekg} and $z_0\neq 0,$ then  there exists a unique univalent function $w=g(z)$ defined on a neighbourhood $U$ of $z_0$,  implicitly defined by the equation  \eqref{woekg}  and $g(z_0) =z_1.$  If $(z_i)_0^{n}$ is a finite orbit {not containing zero,} we denote \begin{equation}\label{gopihqe} g_{z_0, z_1}(z) = g(z)   \ \textrm{and} \   g_{z_0,z_1,\ldots, z_n}(z) = (g_{z_{n-1}, z_{n}} \circ \cdots \circ g_{z_0,z_1})(z). \end{equation} The domain of $g_{z_0,z_1,\ldots, z_n}$ is  an unspecified neighbourhood of $z_0.$ A cycle $(z_i)_0^{n}$  with period $n$ is \emph{repelling}  if the absolute value of the derivative of $g_{z_0, z_1, \ldots, z_n}$ at $z_0$ is strictly greater than $1.$  
 Therefore, no point of a repelling cycle is allowed to be zero.

\begin{defi}[Julia set] The closure of the union of all repelling cycles is the \emph{Julia set} $J(\mathbf{f}_c)$ of the correspondence $\mathbf{f}_c.$ 
\end{defi}

%A subset $\Lambda$ of the punctured plane $\mathbb{C}^*$ is a \emph{hyperbolic repeller} if $\mathbf{f}_c^{-1}(\Lambda) = \Lambda$ and there exists a conformal metric $\rho$ defined on a neighbourhood of $\Lambda$ and $\lambda >1$  such that  $$\|g_{z,w}'(z)\|_{\rho}>\lambda$$ whenever $z$ and $w$ belong  to $\Lambda$ and $w$ is an image of $z.$  Equivalently:

\begin{defi}[Hyperbolic repeller]\label{gjege} A compact set $\Lambda \subset \mathbb{C}^*$ is a {\emph{hyperbolic repeller}} if $\mathbf{f}_c^{-1}(\Lambda)=\Lambda$ and there exist $C>0$ and $\lambda >1$ such that \begin{equation} \label{gjcmwe}|g_{z_0\cdots z_n}'(z_0)| \geq C\lambda^n, \end{equation} for every finite orbit $(z_i)_0^{n}$ contained in $\Lambda.$  \end{defi}

Since every hyperbolic repeller $\Lambda$  is backward invariant, its complement is forward invariant.  { We define   $B_{r}(a) =\{z\in \mathbb{C}: |z-a|<r\}.$ }

\begin{lem}[Koebe's distortion] \label{dgrwbff} Let $\mathcal{F}$ denote the family of all univalent functions $\varphi:B_r(a) \to \mathbb{C},$ where $r>0$ and $a\in \mathbb{C}$ are allowed to vary arbitrarily.  
{There exists a constant $K>0$  such that, for any $\varphi$ in  $\mathcal{F}$,} we have

\begin{equation}\label{igdw} K^{-1} |\varphi'(a)| \cdot |z-w|\leq {|\varphi(z)-\varphi(w)|} \leq K |\varphi'(a)| \cdot |z-w|, \end{equation}
whenever $z$ and $w$ belong to $B_{r/4}(a),$ where $r$ is the radius of the domain $B_{r}(a)$ of $\varphi.$ (Note that $K$ depends neither on $r$ nor $a$).

\end{lem}

\begin{defi}[Expansive constant for correspondences] \normalfont Let $\omega \in \mathbb{N}_{0}\cup \{\infty\}.$ {A pair of orbits $(z_i)_0^{\omega}$ and $(w_i)_0^\omega$ is said to be $\epsilon$-close if $|z_i -w_i|< \epsilon$ for $0\leq i \leq \omega.$ } We say that $\epsilon>0$ is an \emph{expansive constant} for a hyperbolic repeller $\Lambda$ if every pair of $\epsilon$-close orbits $(z_i)_0^\infty$ and $(w_i)_0^\infty$ {in} $\Lambda$  must actually coincide. 
\end{defi}

\begin{lem}\label{gjecb} Every hyperbolic repeller $\Lambda$ of $\mathbf{f}_c$ has an expansive constant $\epsilon>0$ satisfying the following property: if {$(z_i)_0^n$} and $(w_i)_0^n$ is a pair of finite $\epsilon$-close orbits contained in $\Lambda,$ then both maps $(g_{z_0 \cdots z_n})^{-1}$ and $(g_{w_0 \cdots w_n})^{-1}$ are well defined and  coincide on $B_{\epsilon}(z_n).$  In particular, $(g_{z_0 \cdots z_n})^{-1}$  maps $w_n$ to $w_0.$ 
\end{lem} 

\begin{proof} 

Let $X_{n,c}$ denote the space of all finite orbits $(z_i)_0^n$  contained in $\Lambda.$ We know that $X_{n,c}$ is a closed subset of the product space $\Lambda^n$, and therefore $X_n,c$ is a compact metric space with the induced metric from $\Lambda^n$. (We have used this fact repeatedly  in \cite{SS17}; see, for example,  the beginning of \cite[proof of Lemma 2.3]{SS17}).

The proof is divided into four steps.

\noindent  {\bf Claim A:} {\it  given a finite orbit $(z_i)_0^n \in X_{n,c}$ there exist $\epsilon>0$ and $\rho$ with $\epsilon <\rho$ such that for any $(w_i)_0^n$ in the open set $V_0 \subset X_{n,c}$ consisting of all finite orbits which are $\epsilon$-close to $(z_i)_0^n,$ we have:

 $(i)$ the univalent branches $\varphi=g_{z_0 \cdots z_n}$ and $\psi=g_{w_0 \cdots w_n}$ coincide on $B_{\rho}(z_0)$; and 
 
 $(ii)$ the image of $B_{\rho}(z_0)$ under $\varphi$  contains $B_{\epsilon}(z_n).$  Consequently, $\varphi^{-1}$ and $\psi^{-1}$ coincide on $B_{\epsilon}(z_n)$ and both send $w_n$ to $w_0.$ }

 Let us prove the case $n=1.$   Recall that if $z_0\neq 0$, then there exist $q$ univalent branches $g_i: U \to \mathbb{C}$ defined on a neighbourhood of $z_0$ such that $\mathbf{f}_c(U)$ is completely determined by the union of all $g_i(U).$ We may assume, without loss of generality, that for some $\delta>0,$ the diameter $|g_i(U)|$ is less then $\delta$ and $|g_i(z_0) - g_j(z_0)| > 9\delta$, whenever $i \neq j.$ (Roughly speaking, the sets $g_i(U)$ are very small and away from each other).  Hence, for some $0<\epsilon_1< \delta$,  if $w_0 \mapsto w_1$ is $\epsilon_1$-close to $z_0 \mapsto z_1$, it follows that $g_{z_0, z_1}$ and $g_{w_0, w_1}$ must be determined by the same univalent branch $g_{i_0}: U \to \mathbb{C},$ otherwise $z_1$ and $w_1$ would not belong to the same $g_{i_0}(U)$, and consequently, $|z_1-w_1| > \epsilon_1,$ which is impossible, since we have a pair of $\epsilon_1$-close orbits. There exists $\rho>0$ such that $B_{\rho}(z_0) \subset U.$  Hence $g_{z_0, z_1}|_{B_{\rho}(z_0)} = g_{w_0, w_1}|_{B_{\rho}(z_0)}=g_{i_0}|_{B(z_0, \rho)},$ and it follows that $z_1$ belongs to $g_{i_0}(B_{\rho}(z_0)).$ There exists  $\epsilon> 0$  such that $B_{\epsilon}(z_1) \subset g_{i_0}(B_{\rho}(z_0))$, with $\epsilon < \epsilon_1$ and $\epsilon <\rho.$  Then $(i)$ and $(ii)$ of Claim A follow with this choice of $\epsilon$ and $\rho.$ This proves Claim A for  $n=1.$  The general case $n\geq 1$ follows by induction.
  
   Our next step is to eliminate the dependence of $\epsilon$ on $(z_i)_0^n.$ 
 
 \noindent  {\bf Claim B:}  {\it for a fixed $n\geq 1$, there exists $0<\epsilon<1$ such that for any pair of $\epsilon$-close orbits $(z_i)_0^n$ and $(w_i)_0^n$ in $X_{n,c}$,  both maps $(g_{z_0 \cdots z_n})^{-1}$ and $(g_{w_0 \cdots w_n})^{-1}$ are well defined and  coincide on $B_{\epsilon}(z_n).$     }
 
  Using Claim A, we construct a finite  covering $\{V_i\}$ of $X_{n,c}$ and finitely many $ \rho_i > \epsilon_i>0$ for which $(i)$ and $(ii)$ of Claim A hold, whenever $(z_i)_0^n$ and $(w_i)_0^n$ are in $V_i.$ Then chose $0<\epsilon<1$ as being  a Lebesgue number of the covering such that $\epsilon<\epsilon_i$, for every $i$. Any pair of $\epsilon$-close orbits are within the same $V_i;$ by Claim A $(ii)$ the corresponding inverses coincide on $B_{\epsilon_i}(z_n).$ Since $\epsilon<\epsilon_i$, they coincide on $B_{\epsilon}(z_n).$  
 
 (The following $\epsilon$ does not depend on $n$).

\noindent   {\bf Claim C:} {\it there exists  $0<\epsilon<1$ such that, for every $n\geq 1$ and every $(z_i)_0^{n}$ in $X_{n,c}$, the domain of $(g_{z_0 \ldots z_n})^{-1}$ contains $B_{\epsilon}(z_n).$  Moreover, if $(w_i)_0^n$ is $\epsilon$-close to $(z_i)_0^n$ then $(g_{w_0 \cdots w_n})^{-1}$ is defined on $B_{\epsilon}(z_n)$ and coincide with $(g_{z_0 \cdots z_n})^{-1}$ on $B_{\epsilon}(z_n).$    } 

We shall postpone the proof of this claim.

\noindent {\bf Claim D:}  {\it $\epsilon_{0}=\epsilon/4$ is an expansive constant, where $\epsilon$ is given by Claim C. 
}

 { (Therefore, Lemma \ref{gjecb} follows from Claim D with $\epsilon$ replaced by $\epsilon/4.$)}

  {\it Proof of Claim D.} Suppose $(z_i)_0^\infty$ is $\epsilon_{0}$-close to $(w_i)_0^\infty$. We need to show that $z_n=w_n$, for every $n.$ Fix $n$ and let $\{\varphi_m\}$ denote the family of univalent maps $(g_{z_n \cdots z_m})^{-1}$; by Claim C they are all defined on $B_{\epsilon}(z_m),$ for any $m>n$, and each $\varphi_{m}$ maps $w_m$ to $w_n.$  By \eqref{gjcmwe}, the absolute value of the derivative of $(g_{z_n \cdots z_m})^{-1}$ at $z_m$ is at most $C^{-1}\lambda^{-(m-n)}$, where $C>0$ and $\lambda>1$ are the constants of \eqref{gjcmwe}. Using \eqref{gjcmwe} and Koebe's Lemma \ref{dgrwbff} applied to the family of all $\varphi_m:B_{\epsilon}(z_m) \to \mathbb{C}$     we have
  \begin{equation} \label{bjkdfe}
\begin{split}
|z_n -w_n| =|\varphi_m(z_m) - \varphi_{m}(w_m)| & \leq K C^{-1}\lambda^{-(m-n)}|z_m -w_m| \\  & \leq KC^{-1}\lambda^{-(m-n)} 2\epsilon_0 \to 0, 
\end{split}
\end{equation}
  as $m\to \infty.$ Notice that $K$ is the constant of Koebe's Lemma \ref{dgrwbff} and we have used the fact $|z_m -w_m|< \epsilon/4,$ which is fundamental for the application of Lemma \ref{dgrwbff}. We conclude that  $z_n=w_n$, for every $n.$

  {\it Proof of Claim C.}  Let $K$ be the constant of Lemma \ref{dgrwbff}; let $C>0$ and $\lambda>1$ as in  
  \eqref{gjcmwe}. There exists $n>0$ such that $KC^{-1} \lambda^{-n} <1/2.$ Keep $n$ fixed and find a corresponding $\epsilon'>0$ satisfying the properties of Claim B.  Let $k>0$ be an integer. Suppose  $(w_i)_{0}^{nk}$  and $(z_i)_0^{nk}$ are finite orbits contained in $\Lambda$ which are $\epsilon'$-close to one another. By Claim B, for any $0< i \leq k$, the maps $\varphi_i=(g_{z_{n(i-1)} \cdots z_{ni}})^{-1}$ and  $\psi_i=(g_{w_{n(i-1)} \cdots w_{ni}})^{-1}$ are well defined and coincide on $B_{\epsilon'}(z_{ni}).$ Using the same argument of \eqref{bjkdfe} in the proof of Claim D,  we conclude that the maps $\varphi_i$ and $\psi_i$ are well defined contractions by factor $1/2$ on $B_{\epsilon'/4}(z_{ni})$ and $B_{\epsilon'/4}(w_{ni}),$ respectively. Moreover, they define the same inverse branch on  $B_{\epsilon'/4}(z_{ni})$ and  $\varphi_i$ maps $B_{\epsilon'/4}(z_{ni})$  into $B_{\epsilon'/8}(z_{n(i-1)}).$ This ball of radius $\epsilon'/8$ is obviously contained in the domain $B_{\epsilon'}(z_{n(i-1)})$ of $(g_{z_{n(i-2)} \cdots z_{n(i-1)}})^{-1};$ so that we are allowed to perform the compositions:  $$(g_{z_0 \cdots z_{nk}})^{-1}= (g_{z_0 \cdots z_n})^{-1} \circ \cdots  \circ (g_{z_{n(k-1)} \cdots z_{nk}})^{-1},$$       
  $$(g_{w_0 \cdots w_{nk}})^{-1}= (g_{w_0 \cdots w_n})^{-1} \circ \cdots  \circ (g_{w_{n(k-1)} \cdots w_{nk}})^{-1},$$ thereby showing that $(g_{w_0 \cdots w_{nk}})^{-1}$ and  $(g_{z_0 \cdots z_{nk}})^{-1}$ are well defined and coincide on $B_{\epsilon'/4}(z_{nk}).$ Since $k>0$ is arbitrary, Claim C follows with $\epsilon=\epsilon'/4.$

 This completes the proof. \end{proof}

\begin{defi}[LEO hyperbolic repellers] The correspondence $\mathbf{f}_c$ is called \emph{locally eventually onto (LEO)} on a hyperbolic repeller $\Lambda$ if for every  relatively open subset $U$ of $\Lambda$  there exists $n>0$ such that $\mathbf{f}_c^n(U)$ contains $\Lambda. $

\end{defi}
According to Theorem 4.3 of \cite{SS17}, every LEO hyperbolic repeller $\Lambda$ of $\mathbf{f}_c$ is contained in the Julia set $J(\mathbf{f}_c).$

  \begin{defi}[Simple centre]\label{hgdw} The critical point $z=0$ has infinitely many forward orbits under $\mathbf{f}_c.$ If precisely one forward orbit of the critical point of $\mathbf{f}_c$ is periodic and the others diverge to $\infty$, we say that $c$ is a \emph{simple centre.}  
  
  \end{defi}

The following result is stated as Theorem 5.4 in \cite{Rigidity}.

\begin{thmm} If $a$ is a simple centre for the family of holomorphic correspondences $\mathbf{f}_c$, then $J(\mathbf{f}_c)$ is a LEO hyperbolic repeller for every $c$ in a neighbourhood of $a.$

\end{thmm}

We shall now  summarise some results of \cite{SS17}   concerning $\mathbb{C}^2$-extensions.  

\subsection{The $\mathbb{C}^2$-extension} \label{bndhw} Suppose  $J(\mathbf{f}_{c_0})$ is a LEO hyperbolic repeller for the correspondence $\mathbf{f}_{c_0}.$ There exists  a family of holomorphic maps \begin{equation}\label{gjemqe}f_c: V \to \mathbb{C}^2\end{equation} defined on a open subset $V$ of $\mathbb{C}^2$ and parameterised on a neighbourhood $U$ of $c_0,$ such that   the closure of the periodic points of $f_c,$ denoted by $J(f_c)$, is a subset of $\mathbb{C}^2$ which is completely invariant under $f_c.$ (See \cite[Lemma 2.3]{SS17}). Moreover, {$f_c$ is a $p$-to-$1$ map } on $J(f_c).$   The  notation  $J(f_c)$ suggests the definition of a Julia set, or that every periodic point of $f_c$ is repelling.  As a matter of fact, the Jacobian determinant of $f_c^n$ at every periodic point of period $n$ is strictly greater than $1.$ This fact is explained in Remark 2.4 of \cite{SS17} and somehow justifies the notation $J(f_c).$ In spite of this analogy, for technical reasons $J(f_c)$ shall not be referred to as the Julia set of $f_c.$

 The family \eqref{gjemqe} is defined for parameters in a neighbourhood of $c_0$ and {enjoys some important properties.} 

 %Some properties of the family \eqref{gjemqe} stated in \cite{SS17} are: 
 \begin{itemize} \item[$(A)$]
{There exists} a holomorphic motion $h_c: J(f_{c_0}) \to J(f_c)$  parameterised on $U$ and {based at $c_0$,} given by a family of conjugacies $h_c$ from $J(f_{c_0})$ to $J(f_c).$ {(See Theorems B and C of \cite{SS17}).}
 
\item[$(B)$] The projection $\pi(z,w)=z$ establishes a semiconjugacy $\pi: J(f_c) \to J(\mathbf{f}_c)$ {between $f_c$ and $\mathbf{f}_c$,}  in the sense that $J(\mathbf{f}_c)=\pi J(f_c)$ is also a LEO hyperbolic repeller for $\mathbf{f}_c$ and $\pi f_c(x)$ is an image of $\pi(x)$ under $\mathbf{f}_c,$ for every $x$ in $J(f_c).$  Hence $$\pi(x) \to \pi f(x) \to \pi f^2(x) \to  \cdots$$ is a forward orbit under $\mathbf{f}_c.$ This is, by definition, the \emph{projected orbit of} $x.$ {(See Lemma 2.3 of \cite{SS17}).}
 
 %\item[$(C)$] \textcolor{red}{If we consider the space $X_c$ of all forward orbits $(z_i)_0^\infty$ under $\mathbf{f}_c$ which are contained in  $J(\mathbf{f}_c),$  then $g_c(x)= (\pi f_c^n(x))_{n=0}^{\infty}$ defines a homeomorphism from $J(f_c)$ onto $X_c$ which conjugates the dynamics of $f_c$ to the dynamics of  the (left) one-sided shift  $\sigma$ on $X_c.$  }  \textcolor{blue}{(See the proof of Lemma 2.3 of \cite{SS17}). } 
 
 %\item[$(D)$] there exist $r>0$ and $\delta>0$ such that $J(f_c)$ is the Cantor bundle of its projection: $$J(f_c)=\mathcal{B}_{r,\delta}(J(\mathbf{f}_c)),$$ for any $c$ in $U.$

\end{itemize}

\section{The Bowen parameter}

%See \cite{Bowenbook} and \cite{ConformalF} for a detailed reference.    

Recall that a continuous surjective map $f: X \to X $ of a compact metric space is \emph{expanding} if there exists $\ell >1$ such that every point in $X$ has a neighbourhood $V$  such that $f^{-1}(V)$ is a finite union of disjoint open sets $U_j,$ each of which is mapped homeomorphically  onto $V$  and  $$d(f(x), f(y)) > \ell d(x,y)$$  for $x$ and $y$ in $ U_j.$

The LEO property may be defined for every  $f:X \to X.$ It means that every nonempty open set of $X$ is eventually mapped onto $X.$  The following result is stated as Theorem 4.3 in \cite{SS17}.

\begin{thmm}\label{gjegde} {Let $c_0$ be a parameter in $\mathbb{C}$ such that $J(\mathbf{f}_{c_0})$ is a LEO hyperbolic repeller for $\mathbf{f}_{c_0}.$} {For every $c$ in a neighbourhood of $c_0,$}  the map $f_c$ of the family \eqref{gjemqe} {is  LEO on $J(f_c)$ and expanding}  with respect to the metric
$$ d_s(x,y)=  \sum_{n=0}^{\infty} s^{-n}|\pi f_c^n(x) - \pi f_c^n(y) |$$ where $s>1$ is arbitrary  and $x,y$ belong to $J(f_c).$
\end{thmm}

The \emph{dynamic ball of radius $\epsilon$, time $n$ and centre $x$} is defined by  
  $$B(x, n, \epsilon)=\{y\in X: d(f^jx, f^jy)< \epsilon, \ 0 \leq j \leq n\}.$$    Every expanding map $f: X \to X$ has an \emph{expansive constant} $\epsilon>0$ characterised by the fact that if  $d(f^nx,f^ny)<\epsilon$ for every $n\geq 0,$ then $x=y.$  (It is clear that any other positive real number $\epsilon_0 < \epsilon$ is also an expansive constant.) 
  
  A continuous map $f: X \to X$ is \emph{topologically mixing}  if for every pair of nonempty open subsets $U$ and $V$ there exists $n_0$ such that $f^n(U)\cap V$ is nonempty, for every $n\geq n_0.$ LEO maps are  topologically mixing.

   Suppose  $f$ is an expanding map of a compact metric space $(X,d)$, and  $\phi: X\to \mathbb{R}$ is a potential (i.e, a real valued continuous function).   {The topological pressure of $\phi$ with respect to the system $f$ is denoted by $P(\phi),$ and $S_n\phi(x)$ denotes a Birkhoff sum.} Since $f$ is expanding, there exists an expansive constant $\epsilon>0.$ If $\mu$ is a probability measure on $X$ and there exists $C_{\epsilon}>0$ such that $$C_{\epsilon}^{-1} \leq \frac{\mu(B(x,n,\epsilon))}{ \exp (S_n\phi(x) - n P(\phi))} \leq C_{\epsilon}, $$ for every $x$ in $X$ and $n>0,$  then $\mu$ is a \emph{Gibbs measure} of $f$ and $\phi.$ (See \cite[Chap. 4]{ConformalF}).

\subsection{Transfer operators.}  
Let $f: X \to X$ be an expanding map of a compact metric space $(X,d).$ The  \emph{transfer operator} $\mathcal{L}_{\phi}: C(X) \to C(X)$ with potential $\phi$ acts on the space $C(X)$ of continuous complex valued functions $g: X \to \mathbb{C},$ and is defined by $$ \mathcal{L}_{\phi}g(x) = \sum_{f(y)=x} e^{\phi(y)} g(y).$$ 

\noindent (See \cite[Chap. 4]{ConformalF}). The iterate $\mathcal{L}_{\phi}^n$ is precisely the transfer operator $\mathcal{L}_{S_n\phi}$ with respect to $f^n:X\to X.$

 The dual operator $\mathcal{L}_{\phi}^*: M(X) \to M(X)$ acts on the space of complex measures defined on the Borel $\sigma$-algebra of $X.$ It is defined by $\langle \mathcal{L}_{\phi}^*\mu, g\rangle = \langle \mu,  \mathcal{L}_{\phi} g\rangle, $ for every $\mu$ in $M(X)$ and $g$ in $C(X),$ where $\langle \mu, g\rangle = \int_X g  d\mu.$ The following theorem summarises some well known results concerning the transfer operator, and  the first three sentences are often referred to as the \emph{Ruelle-Perron-Frobenius theorem} \cite[Chap. 4]{ConformalF}.

\begin{thmm} \label{gjdcs}  Suppose $f: X \to X$ is an expanding and topologically mixing map of a compact metric space $X.$ If $\phi: X \to \mathbb{R}$ is H\"older continuous and $\lambda_{\phi}=\exp P(\phi),$ then:

\begin{itemize}\item[$(a)$] there exists a unique probability measure $\nu$ on $X$ such that  $\mathcal{L}_{\phi}^*(\nu)=\lambda_{\phi} \nu;$ 

\item[$(b)$] there exists a unique real valued continuous function $h>0$ on $X$ such that $\mathcal{L}_{\phi}(h)=\lambda_{\phi} h$ and $\int_X hd\nu=1;$

\item[$(c)$] for any continuous real valued function $g$ on $X,$  $$\lambda_{\phi}^{-n} \mathcal{L}_{\phi}^{n}(g) \to  h \int_X g d\nu,$$ {as $n\to \infty$, }and the convergence is uniform on $X;$  
\item[$(d)$] the measure $\mu(A) = \int_A hd\nu$  is the unique invariant Gibbs measure of $f$ and $\phi.$

\end{itemize}

\end{thmm}

\begin{remark} \label{qodwe}\normalfont
The number $P(0)$ is also known as the topological entropy of $f.$ 
Suppose that $f: X\to X$ is a $d$-to-$1$ map (i.e., every point has exactly $d$ preimages). Then $f^n$ is $d^n$-to-$1.$ By   Theorem \ref{gjdcs},  $$\lambda_{\phi}^{-n}\mathcal{L}_0^n(1) = \lambda_{\phi}^{-n} \sum_{f^n(y)=x} e^{S_n \phi(y)} \cdot 1 = \lambda_{\phi}^{-n}d^n$$ converges to $h>0$, where $\phi=0.$ It follows that $\lambda_{\phi}= d,$ and the topological entropy of $f$ must be $\log(d).$
\end{remark}

The following result introduces the Bowen parameter $t_c.$ 

Let $z_0(x)$ and $z_1(x)$ denote the first two elements of the projected orbit of an element $x$ of  $J(f_c).$  We define $\varphi_c: J(f_c)  \to \mathbb{R}$ by   $ \varphi_c(x)= -\log|g_{z_0, z_1}'(z_0)|,$ where $z_0$ and $z_1$ are $z_0(x)$ and $z_1(x),$ respectively.   The map $\varphi_c$ is therefore a potential on $J(f_c)$, and  the zero of the pressure function $t\mapsto P(t\varphi_c) $  is defined as the \emph{Bowen parameter $t_c.$} However, we need to show that there exists only one zero of the pressure function. This is established in the following theorem. Note that $\varphi_c(x)$ is defined directly by  \eqref{lkmfg}, for then  $g_{z_0, z_1}'(z_0)$ equals $p(z_1 -z_0)/(qz_0).$

\begin{thmm}[Bowen parameter] \label{igdg}   If the Julia set  $J(\mathbf{f}_c) \subset \mathbb{C}$ of the holomorphic correspondence $(w-c)^q=z^p$ is a LEO hyperbolic repeller, then the topological entropy of the $\mathbb{C}^2$-extension $$f_c: J(f_c) \to J(f_c)$$ is strictly positive. Moreover, the   potential  
\begin{equation} \label{lkmfg}\varphi_{c}(x) = - \log|p(z_1 -z_0)/(qz_0)|\end{equation}   is  H\"older continuous with respect to the metric $d_s$ on $J(f_c),$ and   there exists a unique zero $t_c>0$ of the pressure function $t\mapsto P(t\varphi_c)$ defined on $[0, \infty).$ 
\end{thmm}

\begin{proof} According to  \cite[Chap. 2B]{Bowen1979},  the topological pressure of a potential $\phi$ can be calculated by \begin{equation} \label{eidhg} P(\phi) = \lim_{|\mathcal{C}|\to 0} \lim_{n\to \infty} \frac{1}{n} \log (\inf_{\mathcal{C}_n} \sum_{U\in \mathcal{C}_n} \exp {S_n\phi(U))},\end{equation}

\noindent where $\mathcal{C}$ is any finite covering of $J(f_c);$  $|\mathcal{C}|$ is the greatest diameter of an element of $\mathcal{C};$ and  $S_n\varphi(U)$ is given by the supremum of all $S_n\phi(x)$ when $x\in U.$  The infimum in  \eqref{eidhg} is taken over all possible covers $\mathcal{C}_n$ of $J(f_c)$ whose elements  can be written as $$U=\left \{x\in J(f_c): f_c^{j}(x) \in U_j, \ 0\leq j \leq n\right \},$$ for some finite sequence  $(U_j)_1^n$ of elements of $\mathcal{C}.$

Since $t\mapsto  P(t\varphi_c)$ is continuous, in order to prove the existence of a unique zero it suffices to show that the pressure function  is strictly decreasing with $P(0) >0$ and $P(t\varphi_c) \to -\infty$ as $t\to \infty.$   

%Since  \begin{align}\exp S_n((t+ \tau)\varphi_c)(U) & = \exp S_n(t\varphi_c)(U) \cdot \exp S_n(\tau \varphi_c)(U) &
%\end{align}

\noindent Indeed,  $J(\mathbf{f}_c)$ is a hyperbolic repeller, and therefore $|g_{z_0 \cdots z_n}'(z_0)| \geq C\lambda^n$, where $\lambda>1$ and   $(z_j)_0^{\infty}$ denotes the projected orbit of any $x.$ By Theorem \ref{igdg} and a  simple computation involving the logarithm, 
\begin{equation} \label{nbfwrh}\exp S_n (t\varphi_c)(x) = |g_{z_0\cdots z_n}'(z_0)|^{-t}.\end{equation} Hence,
 \begin{align*} P((t+\tau)\varphi_c) & = \lim_{|\mathcal{C}|\to 0} \lim_{n\to \infty} \frac{1}{n} \log (\inf_{\mathcal{C}_n} \sum_{U\in \mathcal{C}_n} \exp S_n(t\varphi_c)(U) \cdot \exp S_n(\tau \varphi_c)(U)) \\ 
& = \lim_{|\mathcal{C}|\to 0} \lim_{n\to \infty} \frac{1}{n} \log ( \inf_{\mathcal{C}_n} \sum_{U\in \mathcal{C}_n} \sup_{x\in U} |g_{z_0 \cdots z_n}'(z_0)|^{-t} \cdot \exp S_n(\tau \varphi_c)(U)) \\
& \leq  \lim_{|\mathcal{C}|\to 0} \lim_{n\to \infty} \frac{1}{n} \log ( C^{-t} \lambda^{-n t}   \inf_{\mathcal{C}_n} \sum_{U\in \mathcal{C}_n}  \exp S_n(\tau \varphi_c)(U))   \\
&= - t\log(\lambda) + P(\tau \varphi_c). \end{align*}

\noindent Since $\log(\lambda) >0$, we conclude that the pressure function is indeed strictly decreasing and $P(t\varphi_c) \to -\infty$ as $t \to \infty.$  Since $f_c$ is a $p$-to-one map, it follows from Remark \ref{qodwe} that  $P(0) = \log(p)>0.$   Hence, there exists a unique zero $t_c$ of the pressure function in $(0,\infty).$   

   The proof of the H\"older continuity of the potential with respect to the metric $d_s$ is a straightforward application of the mean value theorem. 
   %Recall that $\varphi_c(x) =  - \log |g_{z_0, z_1}'(z_0)|,$ where $(z_j)_0^\infty$ is the projected orbit of $x.$ If $x$ and $y$ are sufficiently close, say $d_s(x,y)< \varepsilon,$ then the projected orbits $(z_j)_0^\infty$ and $(w_j)_0^\infty$ of $x$ and $y$ (resp.) determine the same branch $g_{z_0, z_1}= g_{w_0, w_1}$ in the intersection of their respective domains. Therefore, it suffices to prove H\"older continuity for points $x$ and $y$ which $\varepsilon$-close, for then $\varphi_c(x)$ and $\varphi_c(y)$ are determined by the same branch $g.$ By the mean value theorem,
%$$ |\varphi_c(x) - \varphi_c(y)| = |-\log|g'(z_0)| - \log|g'(w_0)|| \leq C_1|g'(z_0) - g'(w_0)|, $$
%where $C_1$ is an upper bound for the derivative of the logarithm on the interval $[C\lambda, \infty).$ Again, we apply the mean value theorem for $g',$ finding an upper bound for the second derivative of $g$ (this can be done, for instance, using the Cauchy integral formula for the second derivative, which gives an upper bound of $|g''(z)|$ in terms of a uniform bound $|g(z)| \leq M$ of all possible branches $g$). In this way, we conclude that there is another constant $C_2$ such that 
 %$$|\varphi_c(x) -\varphi_c(y)| \leq C_1C_2|z_0 - w_0| \leq C_1C_2 d_s(x,y),$$
% which shows that $\varphi_c$ is actually Lipschitz continuous. 
\end{proof}

\begin{remark} \normalfont {\it The value of $t_c$ does not depend on the particular choice of $f_c.$} 

{In fact,} if we consider the space $X_c$ of all forward orbits $(z_i)_0^\infty$ under $\mathbf{f}_c$ which are contained in  $J(\mathbf{f}_c),$  then $g_c(x)= (\pi f_c^n(x))_{n=0}^{\infty}$ defines a homeomorphism from $J(f_c)$ onto $X_c$ which conjugates the dynamics of $f_c$ to the dynamics of  the (left) one-sided shift  $\sigma$ on $X_c$  \cite[pages 3110-3112]{SS17}.  Since the pressure is a topological invariant, this is sufficient to show that $t_c$ depends only on the dynamics of the shift map on the space of orbits.
   
Indeed,  \begin{equation}\label{ghewdgd} P(t\varphi_c; f_c) = P(t\varphi_c \circ g_c^{-1}; \sigma).\end{equation}  The notation $P(\phi; f)$ makes explicit the dependence of the pressure on the dynamics of $f.$  By definition, $g_c^{-1}$ maps every $(z_i)_0^\infty$ in $X_c$ to the unique $x\in J(f_c)$ such that $\pi f_c^n(x)=z_n,$ for every $n.$  Since $\varphi_c$ is defined by \eqref{lkmfg} and $(z_i)_0^{\infty}$ is the projected orbit of $x,$
$$(t\varphi_c\circ g_c^{-1})\left((z_i)_{0}^\infty\right)= -\log | p(z_1 -z_0)/(qz_0)|.$$  This shows that the potential  on the right side of \eqref{ghewdgd}   does not depend on $g_c$ {or}  $f_c.$ Hence the same is true for the unique zero of the pressure function, as desired.  \end{remark}
\section{Hausdorff dimension}

The \emph{$s$-dimensional Hausdorff outer measure} $\mathcal{H}^s$ of a set $\Lambda \subset \mathbb{C}$ is defined by

$$\mathcal{H}^s (\Lambda)= \lim_{\delta \to 0}  \inf \sum_{i=1}^{\infty} |U_i|^s,$$ where $\inf$ is taken over all countable coverings $\{U_i\}_0^{\infty}$ of $\Lambda$ with diameter $|U_i|\leq \delta.$
There exists a unique  nonnegative real number $d$ characterised by the following properties:  $\mathcal{H}^s(\Lambda) = 0$ if $s>d$  and $\mathcal{H}^s(\Lambda) =\infty$ if $0\leq s <d.$ The number $d$ is, by definition, the \emph{Hausdorff dimension} of $\Lambda, $ denoted by $\operatorname{dim_{H}} \Lambda.$

\begin{thmm}\label{kfieg} Suppose $J(\mathbf{f}_c)$ is a LEO hyperbolic repeller and let $t_c$ be the unique zero of the associated pressure function. {Then  $\mathcal{H}^{t_c}(J(\mathbf{f}_c))$ is finite,} and therefore $$\operatorname{dim_{H}} J(\mathbf{f}_c) \leq t_c .$$ 

\end{thmm}

\begin{proof} {Fix a parameter $c$ for which {$\Lambda=J(\mathbf{f}_c)$} is a hyperbolic repeller. By Lemma \ref{gjecb}, $\Lambda$ has an expansive constant $\epsilon>0.$ }  Cover $J(f_c) \subset \mathbb{C}^2$ with finitely many dynamic balls  of fixed time $n$ and radius $\epsilon:$ 
  $$B(x_i, n, \epsilon) =\{ y\in J(f_c):  d_s(f_c^j(x_i), f_c^{j}(y))< \epsilon, \ 0\leq j \leq n \}.$$

%such that $J(\mathbf{f}_c)$ is a LEO hyperbolic repeller and $4\epsilon$ is an expansive constant for $\mathbf{f}_c.$ \textcolor{red}{By  Koebe's Lemma \ref{dgrwbff}, }there exists a constant  $K$ such that $\varphi=(g_{z_0\cdots z_n})^{-1}$ satisfies \eqref{igdw}  on $B_{\epsilon}(z_n)$ with $r=4\epsilon$ and $a=z_n.$

  The finite covering is possible because $J(f_c)$ is compact. The covering is called \emph{minimal} if   $x_j \not \in B(x_i, n, \epsilon)$ whenever $i\neq j.$
  We may assume  that the covering is minimal by removing $B(x_i, n,  \epsilon)$ if $x_i$ is contained in some other $B(x_j, n, \epsilon).$  Since the covering is minimal, the corresponding dynamic balls with radius $\epsilon/2$  are pairwise disjoint, for if $B(x_i, n, \epsilon/2)$ intersects $B(x_j, n, \epsilon/2),$ then by the triangle inequality  $x_j$ must belong to $B(x_i, n, \epsilon).$ {But   $x_j$ should not belong to $B(x_i, n, \epsilon) $ when $i\neq j$ (the covering is minimal). } 

It will be convenient to denote such covering of $J(f_c)$  by $\mathcal{V}_n$, and its elements by $V_i= B(x_i, n, \epsilon).$ Now we are going to construct a covering of $J(\mathbf{f}_c).$ 

\vspace{0.1cm}
\noindent {\bf Claim:} {\it the family of  all sets $U_i = (g_{z_0\cdots z_n})^{-1}(B_{\epsilon}(z_n)), $ where $(z_k)_{0}^{\infty}$ is the projected orbit  of the centre $x_i$ of $V_i$,  is a covering of $J(\mathbf{f}_c)$. We shall denote this covering by $\mathcal{U}_n.$  }
\vspace{0.1cm}

{By Lemma \ref{gjecb}, each set $U_i$ is well defined since the domain of $(g_{z_0\cdots z_n})^{-1}$ contains $B_{\epsilon}(z_n).$}    As we shall see, the sets $U_i$ do in fact cover $J(\mathbf{f}_c)$ because each $U_i$ contains $\pi V_i$, $J(\mathbf{f}_c)$ is the projection of $J(f_c)$, and $\{V_i\}$ is a covering of $J(f_c).$ In order to check that $\pi V_i$ is contained in $U_i$ it suffices to show that $\pi(y)$ is in $U_i,$ whenever $y\in V_i.$  Since $y$ is in $B(x_i, n, \epsilon),$ by definition the distance with respect to the metric $d_s$ between $f_c^{j}(x_i)$ and $f_c^{j}(y)$ is strictly less than $\epsilon$, for $0\leq j \leq n.$ This means 
$$\sum_{k=0}^{\infty} s^{-k} |\pi f_c^{k+j}(x_i) - \pi f_c^{k+j}(y) | < \epsilon, $$ 
for $0\leq j\leq n.$ In particular, the first term of the above series is less than $\epsilon$, and   it follows that $|z_j  - w_j| < \epsilon$ for $0\leq j \leq n,$ where $(z_j)_0^{\infty}$ and $(w_j)_0^{\infty}$ are the projected orbits  of $x_i$ and $y$, respectively. 
{By Lemma \ref{gjecb}, $(g_{z_0 \cdots z_n})^{-1}$ maps $w_n$ to $w_0.$ Hence $w_0 =\pi(y)$ belongs to $U_i,$ as desired. }

The diameter of $\mathcal{U}_n$ is defined as the supremum of all diameters $| U_i |.$ 

\vspace{0.1cm}
\noindent {\bf Claim}:  {\it the diameter of $\mathcal{U}_n$ tends to zero as $n\to \infty.$} 
\vspace{0.1cm}

\noindent By  Definition \ref{gjege}  and Koebe's Lemma \ref{dgrwbff}, for every element $U_i$ of $\mathcal{U}_n$ we have 
$$|U_i| = |(g_{z_0\cdots z_n})^{-1}(B_{\epsilon}(z_n))| \leq K |g_{z_0 \cdots z_n}'(z_0)|^{-1} 2\epsilon \leq  (2\epsilon K C^{-1}) \lambda^{-n} \to 0,    $$
as $n\to \infty.$ The claim is proved.

By Theorem \ref{gjegde}, $f_c:J(f_c) \to J(f_c)$ is LEO and expanding. Therefore, for every H\"older continuous potential there corresponds a unique invariant Gibbs measure. Let $\phi=t_c\varphi_c.$ Since $P(\phi)=0$ and $\epsilon/2$ is also an expansive constant, the corresponding Gibbs measure $\mu$ for the system  satisfies:
\begin{equation} \label{jhjewer} {C_{\epsilon/2}^{-1}\leq \frac{\mu(B(x,n,\epsilon/2))}{\exp {S_n\phi(x)}} \leq C_{\epsilon/2},} \end{equation} for any $x$ in $J(f_c)$ and $n,$ where $C_{\epsilon/2}$ is a constant independent of $x$ and $n.$ {By Koebe's Lemma \ref{dgrwbff}, \eqref{nbfwrh} and \eqref{jhjewer}  we have}

\begin{equation}
\begin{split}
 |U_i|^{t_c} = |(g_{z_0\cdots z_n})^{-1}(B_{\epsilon}(z_n)) |^{t_c} & \leq (K |g_{z_0 \cdots z_n}'(z_0)  |^{-1}2\epsilon)^{t_c} \\  & \leq (2\epsilon K)^{t_c}  C_{\epsilon/2} \mu(B(x_i, n, \epsilon/2)),
\end{split}
\end{equation}

\noindent for every element $U_i$ of $\mathcal{U}_n.$ Since $\mu$ is a probability measure and the dynamic balls $B(x_i, n, \epsilon/2)$ are pairwise disjoint, we conclude that $ \sum_{i=0}^{\infty} |U_i|^{t_c} \leq (2\epsilon K)^{t_c}C_{\epsilon/2}. $  Since $\mathcal{U}_n$ is a covering of $J(\mathbf{f}_c)$ whose diameter  tends to zero as $n\to \infty,$ the $t_c$-dimensional Hausdorff measure of $J(\mathbf{f_c})$ is finite. Hence, $\operatorname{dim_{H}}J(\mathbf{f}_c) \leq t_c.$
\end{proof}

\begin{cor} \label{abc} For any parameter $c$ sufficiently close to a simple centre,
$$\operatorname{dim_{H}} J(\mathbf{f}_c) \leq t_c. $$

\end{cor}

\begin{proof}
By Theorem 5.4 of \cite{Rigidity}, the Julia set $J(\mathbf{f}_c)$ is a LEO hyperbolic repeller, for every $c$ in a certain subset $H_{\beta}$ of the parameter space. By definition 4.1 of \cite{Rigidity}, the set $H_{\beta}$ includes a sufficiently small neighbourhood of every simple centre. \end{proof}   

\begin{cor} \label{nmcs} Suppose $J(\mathbf{f}_c)$ is a LEO hyperbolic repeller and  $|g'_{z,w}(z) | \geq  \kappa >1,$ for every $z,w$ in $J_c$ such that $w$ is an image of $z$ under $\mathbf{f}_c.$ Then \begin{equation} \label{dgew} \operatorname{dim_{H}} J(\mathbf{f}_c)  \leq (\log p)/(\log \kappa). \end{equation}

\end{cor}

\begin{proof}  Let $\phi= t_c \varphi_c.$ Since $P(\phi)=0$, by the Ruelle-Perron-Frobenius Theorem \ref{gjdcs}, $ \mathcal{L}_{\phi}^{n}(1) \to h$ uniformly on $J(f_c),$ where $h$ is a continuous function from $J(f_c)$ to $(0,\infty).$ 
By \eqref{nbfwrh}, we have  \begin{equation} \label{gherw}\mathcal{L}_{\phi}^n(1)(x) = \sum_{f^n_c(y)=x} e^{S_n\phi(y)}\cdot 1 = \sum_{f_c^n(y)=x} |g_{w_0\cdots w_n}'(w_0)|^{-t_c} \leq  p^n \kappa^{-nt_c},\end{equation}
where $(w_i)_0^{\infty}$ is the projected orbit of $y.$ In the above estimate we have used the hypothesis and the fact that  $f_c^n$ is a $p^n$-to-$1$ map. Since $\mathcal{L}_{\phi}^n(1)$ converges to a positive function, we have $p\kappa^{-t_c} \geq 1,$ otherwise $p^n\kappa^{-n t_c}$ would converge to zero.
 Solving the inequality $p\kappa^{-t_c} \geq 1$ for $t_c$ yields \eqref{dgew}. \end{proof}

\begin{cor}[\bf Zero area] \label{hjeeh}If $q^2<p$ and $c$ is sufficiently close to zero, then   \begin{equation}\label{qpde} \operatorname{dim_{H}} J(\mathbf{f}_c) < 2.\end{equation}

 \end{cor}

\begin{proof}   The derivative of a branch $g(z)$ of $\mathbf{f}_c$ is given by $(p/q)(w-c)/z$ where $w=g(z).$ {The set function $c\mapsto J(\mathbf{f}_c)$  is continuous with respect to the Hausdorff topology for parameters $c$ in an open set $H_{\beta}$  which contains zero and every simple centre (see Theorem 5.11 of \cite{Rigidity}).}  Since $J(\mathbf{f}_0)$ is the unit circle $\mathbb{S}^1,$ we conclude that $|g_{z,w}'(z)| \geq \kappa$ for $z$ and $w$ in $J(\mathbf{f}_c)$, where  $\kappa>1$ can be chosen arbitrarily close to $p/q$ as $c$ tends to zero. It is easy to check that $\log(p)/\log (p/q) <2$ if, and only if, $q^2<p.$ In particular, $\log(p)/\log(\kappa) <2$ for $\kappa$ close to $p/q.$ From Corollary \ref{nmcs} we conclude \eqref{qpde}, for every $c$ sufficiently close to zero. 
\end{proof}

\begin{remark} \normalfont {\it For arbitrary integers $p>q\geq 2$, we have $t_0 >1,$ and therefore the estimate of Theorem \ref{kfieg} is not sharp if $c=0.$ }

{\it Proof.} Indeed,  in \eqref{gherw} $|g'_{w_0 \cdots w_n}(w_0)|$ equals $(p/q)^n$  when all the elements of the projected orbit $(w_i)_0^{\infty}$ are contained in  $\mathbb{S}^1.$  It follows from \eqref{gherw} that $\mathcal{L}_{\phi}^n(1)$ is simply $(p/q)^{-nt_0}p^n$ when $\phi$ is the potential $t_0\varphi_0.$ In particular, $\mathcal{L}_{\phi}(1)$ is a constant function $\omega \cdot 1$, where $\omega$ is a real number and $1$ is the function which is constantly one on $J(f_0).$ Since $\mathcal{L}_{\phi}^n(1)$ is $\omega^n\cdot 1$, which converges to a positive function $h>0$ on $J(f_0)$, it follows that $\omega=1$ and consequently $h=1.$ Since $\mathcal{L}_{\phi}(1)$ is $(p/q)^{-t_0}p$ and $1$ is a fixed point of $\mathcal{L}_{\phi}$, it follows that $(p/q)^{-t_0}p=1.$ Since $p>q\geq 2$, we have $t_0 >1$. 
\end{remark}

\begin{cor} \label{abcxz}If every orbit of zero  under $\mathbf{f}_c$ diverges to infinity, then $$\dim_{H} J(\mathbf{f}_c) \leq t_c.$$
\end{cor}

\begin{proof}  According to \cite[Theorem 5.4]{Rigidity}, if every orbit of zero under $\mathbf{f}_c$ diverges to infinity, then $J(\mathbf{f}_c)$ is a LEO hyperbolic repeller. It follows from Theorem \ref{kfieg} that $\dim_{H} J(\mathbf{f}_c) \leq t_c.$  \end{proof}

\subsection*{Acknowledgments}
Research partially supported by the grants 2016/16012-6 S\~ao Paulo Research Foundation and CNPq 232706/2014-0.
  The author would like to thank Daniel Smania for many discussions and key insights which led to some preliminary results of this paper, already described in the author's PhD thesis \cite{Carlos}.

\bibliographystyle{amsplain}
\bibliography{oi}

%\begin{thebibliography}{9}

%\bibitem{lamport94}
  %Leslie Lamport,
  %\emph{\LaTeX: a document preparation system},
  %Addison Wesley, Massachusetts,
  %2nd edition,
  %1994.
%\bibitem{Sold}
%S??lodkowski, Zbigniew. Approximation of analytic multifunctions. Proc. Amer. Math. Soc. 105 (1989),
%no. 2, 387?396.

%\bibitem{Lyubich} Dujardin, Romain; Lyubich, Mikhail Stability and bifurcations for dissipative polynomial automorphisms of ?2. Invent. Math. 200 (2015), no. 2, 439?511. 37F45
%\end{thebibliography}

\end{document}